\documentclass[reqno,11pt]{amsart}

\usepackage{geometry}
\usepackage[usenames,dvipsnames]{xcolor}
\usepackage[shortlabels]{enumitem}
\usepackage{lineno}
\usepackage{marginnote}
\usepackage[abbrev,msc-links,backrefs]{amsrefs}

\usepackage[T1]{fontenc}
\usepackage[english]{babel}
\usepackage[utf8]{inputenc}

\usepackage{hyperref}
\usepackage{url}
\hypersetup{%
    backref = page,%
    colorlinks,%
    linkcolor = {red!60!black},%
    citecolor = {green!60!black},%
    urlcolor = {blue!60!black}%
}

\usepackage{graphicx}
\graphicspath{{../figures/}}
\usepackage{tikz}

\usepackage{amsmath}
\usepackage{amssymb}
\usepackage{amscd}
\usepackage{amsthm} 

\theoremstyle{plain}
\newtheorem{thm}{Theorem}[section]
\newtheorem{lemma}[thm]{Lemma}

\newtheorem{prop}[thm]{Proposition}

\newtheorem{conj}[thm]{Conjecture}

\theoremstyle{remark}

\newcommand{\defi}[1]{\emph{\color{red!50!black}#1}}

\let\eps=\varepsilon
\let\theta=\vartheta
\let\phi=\varphi

\def\dlin{d^{lin}}
\def\davg{d^{avg}}
\def\deg{d}

\def\triangle{\tikz{
    \centering
    \begin{scope}[scale=0.2]
    \node[draw,shape=circle,fill=black,inner sep=0pt,minimum size=2pt] (u) at (-0.4,-0.2) {};
    \node[draw,shape=circle,fill=black,inner sep=0pt,minimum size=2pt] (v) at (0.4,-0.2) {};
    \node[draw,shape=circle,fill=black,inner sep=0pt,minimum size=2pt] (x) at (0,0.5) {};
    \draw[thick] (u) -- (v);
    \draw[thick] (v) -- (x);
    \draw[thick] (x) -- (u);
    \end{scope}
}}

\def\cherry{\tikz{
    \centering
    \begin{scope}[scale=0.2]
    \node[draw,shape=circle,fill=black,inner sep=0pt,minimum size=2pt] (u) at (-0.4,-0.2) {};
    \node[draw,shape=circle,fill=black,inner sep=0pt,minimum size=2pt] (v) at (0.4,-0.2) {};
    \node[draw,shape=circle,fill=black,inner sep=0pt,minimum size=2pt] (x) at (0,0.5) {};
    \draw[thick] (v) -- (x);
    \draw[thick] (x) -- (u);
    \end{scope}
}}

\begin{document}

\title{The Brown-Erd\H{o}s-S\'os conjecture in dense triple systems}

\author[G. Santos]{Giovanne Santos}
\address{Departamento de Ingeniería Matemática, Universidad de Chile, Santiago, Chile}
\email{gsantos@dim.uchile.cl}

\author[M. Tyomkyn]{Mykhaylo Tyomkyn}
\address{Department of Applied Mathematics (KAM MFF), Charles University, Prague, Czech Republic}
\email{tyomkyn@kam.mﬀ.cuni.cz}
\thanks{GS has been supported by ANID/Doctorado Nacional/21221049. MT has been supported by GA\v{C}R grant 25-17377S and ERC Synergy Grant DYNASNET 810115.}

\maketitle

\begin{abstract}
The famous Brown-Erd\H{o}s-S\'os conjecture from 1973 states, in an equivalent form, that for any fixed $\delta>0$ and integer $k\geq 3$ every sufficiently large linear $3$-uniform hypergraph of size $\delta n^2$ contains some $k$ edges spanning at most $k+3$ vertices. We prove it to hold for $\delta>4/5$, establishing the first bound of this kind. 
\end{abstract}

\section{Introduction}
\label{sec:intro}

One of the central problems in extremal hypergraph theory, the notoriously difficult Brown-Erd\H{o}s-S\'os conjecture (BESC, for short) from 1973~\cite{BES71} states that
for every $\delta > 0$
and $k \geq 3$ there exists an integer $n_0$ such that 
every $3$-uniform hypergraph on $n\geq n_0$ vertices with at
least $\delta n^{2}$ edges contains a $(k+3,k)$-configuration, i.e.~a set
of~$k$ edges containing in their union at most~${k + 3}$ vertices. 
Its first case $k=3$ was proved by Rusza and Szemerédi~\cite{RS76} in what
became known as the $(6,3)$-theorem --- an influential result in its own right, with far-reaching consequences
in extremal graph theory, additive combinatorics and graph property testing. Its proof featured one of the first applications of Szemerédi's
regularity lemma, and it is believed that a proof of further cases of the conjecture, let alone of the BESC in full, would likewise lead to important new insights. However, despite a lot of effort, the conjecture remains open for all~${k \geq 4}$. 

It is well-known (see e.g.~Claim 1 in~\cite{S15})
that the BESC reduces to the case of~\defi{linear} hypergraphs, i.e.~$3$-uniform
hypergraphs ($3$-graphs, for short) where any two edges share at most one
vertex. Given a linear $3$-graph~$\mathcal{H}$ with $n$ vertices and $m$
edges, define the \defi{linear density} of~$\mathcal{H}$
as~${\dlin(\mathcal{H})=3m/\binom{n}{2}}$.

\begin{conj}[BESC restated]
	\label{conj:besclinear}
	For every $k \geq 4$ and $0<\delta\leq 1$ there exists~$n_0 = n_0(k,\delta)$ such that every linear $3$-graph $\mathcal{H}$ with
  $n \geq n_0$ vertices and~${\dlin(\mathcal{H}) \geq \delta}$
  contains a~${(k+3,k)}$-configuration.
\end{conj}

It is easy to see that the above statement holds for $\delta=1$, i.e.~in
complete (Steiner) triple systems -- a desired configuration is produced by a
simple greedy algorithm. Similarly, it is not hard to deduce for any
given $k$ the existence of $\delta=\delta(k)<1$ that guarantees
a $(k+3,k)$-configuration. However, it seems less straightforward to
prove the conjecture statement for a fixed $\delta<1$ and all $k$, and we were not able to find such a result in the literature. Our aim in this note is to close
this gap by showing that any $\delta>4/5$ would suffice. 
\begin{thm}
  \label{thm:main}
  For every $k \geq 4$ and $\eps > 0$ there exists
  $n_0 = n_0(k,\eps)$ such that any linear $3$-graph~$\mathcal{H}$ with
  $n \geq n_0$ vertices and~${\dlin(\mathcal{H}) \geq 4/5+\eps}$ contains
  a~${(k+3,k)}$-configuration.
\end{thm}

Our proof combines an analysis of the bow-tie graph from the works of Shapira and the second author~\cite{ST21}, and Keevash and Long~\cite{KL20}, with a Goodman-type~\cite{Goodman} inequality between subgraph counts.  

\section{Preliminaries}
\label{sec:pre}\label{sec:bowtie}

Given a graph $G=(V,E)$, we use $|G|$ and $e(G)$ to denote $|V|$ and $|E|$, respectively. 
The \defi{average degree} of $G$ is $\davg(G)=2e(G)/|G|$. We also write $e(\mathcal{H})$ for the number of edges of a $3$-graph $\mathcal{H}$.

We denote by $\kappa_{\triangle}(G)$ and $\kappa_{\cherry}(G)$ the number of
triangles and `cherries' in $G$, respectively~---~a cherry is a (subgraph) copy of
the $3$-vertex path. We use the following well-known inequality between these two quantities (see Chapter VI.1 in~\cite{B78-extremal}). 
For completeness, we include its short proof below.

\begin{lemma}
\label{lemma:triangles}
  If $G$ is a graph on $n$ vertices,
  then
  \[
    3\kappa_{\triangle}(G) \geq 2\kappa_{\cherry}(G) - e(G)(n-2).
  \]
\end{lemma}

\begin{proof}
  Let $p_1$ and $p_2$ denote the number of induced subgraphs of $G$ with $3$
  vertices and exactly $1$ and $2$ edges, respectively. Since
  every edge of $G$ appears in exactly $n-2$ induced $3$-vertex subgraphs of $G$, double counting yields 
  \begin{equation*}
    e(G)(n-2) = 3\kappa_{\triangle}(G) + 2p_2 + p_1.
  \end{equation*}
  On the other hand, counting the cherries in induced $3$-vertex subgraphs of $G$ gives
  \begin{equation*}
    \kappa_{\cherry}(G) = 3\kappa_{\triangle}(G) + p_2,
  \end{equation*}
In combination, we obtain
$$2\kappa_{\cherry}(G) - e(G)(n-2)=6\kappa_{\triangle}(G)+2p_2-3\kappa_{\triangle}(G)-2p_2-p_1=3\kappa_{\triangle}(G)-p_1\leq 3\kappa_{\triangle}(G).
$$
\end{proof}
Next, we recall the definition and properties of the bow-tie
graph of a linear $3$-graph.
This notion was introduced in~\cite{ST21} in the context of a Ramsey version of the BESC, and used subsequently by Keevash and Long~\cite{KL20} to study the BESC in hypergraphs of high uniformity.  

Let $\mathcal{H} = (V,E)$ be a linear $3$-graph.
The \defi{bow-tie graph} $B_{\mathcal{H}}$ of $\mathcal{H}$ is defined as
follows. The vertices of $B_{\mathcal{H}}$
are all unordered pairs of edges $\{e,f\}$ in $E$ such that $|e \cap f| = 1$.
The edge set of~$B_{\mathcal{H}}$ is defined as
\[
  E(B_{\mathcal{H}}) = \{ \{e,f\},\{f,g\} : |e\cap f| = |f \cap g| = |e\cap g| = 1, |e\cap f\cap g| = 0\}.
\]
\noindent
The~\defi{underlying graph}~$U_{\mathcal{H}}$ of $\mathcal{H}$ is defined to have the same
vertex set~$V(U_{\mathcal{H}}) = V$, and to have an edge between vertices~$u$
and~$v$ if and only if~${\{u,v\} \subseteq e}$ for some~${e \in E}$.
We will omit the subscripts and write $B$ and $U$ when $\mathcal{H}$ is clear from the context. 

There is a direct connection between the order and size of $B$, and the subgraph counts in $U$.

  \begin{lemma} \label{obs:upper_vtx_b}
$4|B_\mathcal{H}| \leq \kappa_{\cherry}(U_\mathcal{H})$.
  \end{lemma}
  \begin{proof}
Let $C$ be the set of all cherries in~$U$ whose vertex set does not coincide with an edge of~$\mathcal{H}$. Since~$\mathcal{H}$ is
    linear, the vertex set of each cherry in $C$ is a subset of the vertex set of a unique pair of intersecting edges of~$\mathcal{H}$.
    Conversely, each pair of intersecting edges of $\mathcal{H}$ gives rise to exactly $4$ cherries in $C$. Therefore,
    \[
      4|B| = |C| \leq \kappa_{\cherry}(U).
    \]
  \end{proof}
\noindent 
Observing that edges of $B$ are closely related to $(6,3)$-configurations in $U$ gives the following. 
\begin{lemma}[\cite{ST21}, Proposition 2.1, Remark 2.2]\label{lemma:eB}
We have \begin{enumerate}
\item[(i)]	$\Delta(B_{\mathcal{H}}) \leq 8$, 
\item[(ii)] $e(B_\mathcal{H}) = 3\kappa_{\triangle}(U_\mathcal{H}) - 3e(\mathcal{H})$.
	\end{enumerate}
\end{lemma}

\noindent
Combining the above lemmas, we can bound the average degree of the bow-tie graph as follows.

\begin{lemma}
\label{lemma:avg_degree}
  Let $0 < \delta \leq 1$ and let $\mathcal{H}$ be a linear $3$-graph on $n$ vertices
  with $\dlin(\mathcal{H}) =d\geq \delta$. Then
  \[
    \davg(B_{\mathcal{H}}) \geq 16 - \frac{8n}{\delta(n-1)-1}.
  \]
 In particular, if $\delta=4/5+\eps$ for some $\eps>0$ and $n\geq n_3(\eps)$, then 
 $$\davg(B_{\mathcal{H}})\geq 6+\eps.$$ 
\end{lemma}

\begin{proof}
  By Lemma~\ref{lemma:eB}(ii) and Lemma~\ref{lemma:triangles}, we have
  \begin{align*}
    e(B)&=3\kappa_{\triangle}(U) - 3e(\mathcal{H}) \\
    &\geq 2\kappa_{\cherry}(U)-e(U)(n-2) - 3e(\mathcal{H})\\
    &= 2\kappa_{\cherry}(U) - 3e(\mathcal{H})(n-2) - 3e(\mathcal{H})\\
    &\geq 2\kappa_{\cherry}(U) - 3e(\mathcal{H})n.
  \end{align*}
  \noindent
By this and Lemma~\ref{obs:upper_vtx_b}, we obtain
  \[
    \davg(B) = \frac{2e(B)}{|B|} \geq \frac{4\kappa_{\cherry}(U)-6e(\mathcal{H})n}{\kappa_{\cherry}(U)/4} = 16 - 24\frac{e(\mathcal{H})n}{\kappa_{\cherry}(U)} = 16 - 4\frac{d n^2(n-1)}{\kappa_{\cherry}(U)}. 
  \]
  Note that, by Jensen's inequality, 
  \[
    \kappa_{\cherry}(U) = \sum_{v \in V(U)}\binom{\deg_U(v)}{2} \geq n \binom{\frac{1}{n}\sum_{v \in V(U)}\deg_U(v)}{2} = n\binom{\frac{2e(U)}{n}}{2} = n\binom{\frac{6e(\mathcal{H})}{n}}{2}.
  \]
\noindent
  Since $6e(\mathcal{H}) = 2d\binom{n}{2} = \delta n(n-1)$, it follows that
  \begin{align*}
    \kappa_{\cherry}(U) \geq  n\binom{\frac{6e(\mathcal{H})}{n}}{2}
                        = n\binom{d(n-1)}{2}
                        = \frac{d^2n(n-1)^2}{2} - \frac{d n(n-1)}{2}.
  \end{align*}
  Therefore,
  \begin{align*}
    \davg(B) &\geq 16 - 4\frac{d n^2(n-1)}{\kappa_{\cherry}(U)} \geq 16 - \frac{8d n^2(n-1)}{d^2n(n-1)^2-d n(n-1)} = 16 - \frac{8n}{d(n-1)-1}\\
    &\geq 16 - \frac{8n}{\delta(n-1)-1}.
  \end{align*}
  For the second assertion, substitute $\delta=4/5+\eps$, and suppose that 
    \[
  \frac{n_3-1}{n_3} \geq 1 - \frac{\eps}{4+5\eps}\;\;\text{ and }\;\;\frac{5}{4n_3} \leq \frac{\eps}{2}.
  \]
  We obtain
    \begin{align*}\label{eq:l_avgdegree}
  	\davg(B) &\geq 16 - \frac{8n}{(4/5+\eps)(n-1)-1}\\
    &= 16 - \frac{10}{(1+5\eps/4)(n-1)/n-5/(4n)}\\
  	&\geq 16 - \frac{10}{(1+5\eps/4)(1-\eps/(4+5\eps)) - \eps/2}\\
  	&= 16 - \frac{10}{((4+5\eps)/4)((4+4\eps)/(4+5\eps)) - \eps/2}\\
  	&= 16 - \frac{10}{1+\eps - \eps/2}\\
  	&= \frac{6 + 8\eps}{1+\eps/2}\\
  	&\geq 6 + \eps.
  \end{align*}
\end{proof}
\noindent
We shall also need the following lower bound on $|B|$. A similar bound was used in~\cite{KL20}.

\begin{lemma}
	\label{lemma:bowties}
	For every $0 < \delta \leq 1$ there exists~$n_1 = n_1(\delta)$ such
	that the following holds for all~${n \geq n_1}$.
	If~$\mathcal{H}$ is
	a linear $3$-graph on~$n$
	vertices with~${\dlin(\mathcal{H}) \geq \delta}$, then
	\[
	|B_\mathcal{H}| \geq \frac{\delta^2}{16}\,n^3.
	\]
\end{lemma}

\begin{proof}
	Put $n_1 = 12\delta^{-1}$. Suppose $\mathcal{H}$ is a linear $3$-graph on $n\geq n_1$ vertices
	with~$\dlin(\mathcal{H}) \geq \delta$.
	We count the pairs of intersecting edges in~$\mathcal{H}$.
	By Jensen's inequality and the choice of $n_1$ we have
	\[
    |B| = \sum_{v \in V(\mathcal{H})}\binom{\deg_{\mathcal{H}}(v)}{2} \geq n\binom{\frac{1}{n}\sum_{v \in V(\mathcal{H})}\deg_{\mathcal{H}}(v)}{2} = n\binom{\frac{3e(\mathcal{H})}{n}}{2} \geq n\binom{\frac{\delta\binom{n}{2}}{n}}{2} = n\binom{\frac{\delta(n-1)}{2}}{2} \geq \frac{\delta^2 n^3}{16}.
  \]
\end{proof}

When the bow-tie graph of a linear $3$-graph $\mathcal{H}$ has a
large (connected) component, a 
$(k+3,k)$-configuration in $\mathcal{H}$ can be constructed inductively by following a long path inside the component.
\begin{prop}[\cite{ST21}, Lemma 2.3]
  \label{lemma:large_config}
  Let $k \geq 3$, and let $\mathcal{H}$ be a linear $3$-graph
  on $n$ vertices. If $B_\mathcal{H}$ has a component
  with at least~$3^{10k^2}$ vertices, then $\mathcal{H}$ has
  a~$(k+3,k)$-configuration.
\end{prop}
We say that a component $C$ of~$B$ is \defi{dense} if~${\davg(C) \geq 6}$. If the bow-tie graph has many dense components
of bounded size, we are also able to find a~$(k+3,k)$-configuration. The strategy used in this case is, roughly speaking, to find small configurations with many edges in each component in such a way that together they form a~$(k+3,k)$-configuration.

\begin{prop}[\cite{ST21}, Lemmas 3.6 and 3.7]
  \label{lemma:dense_config}
  For every $k \geq 3$ and $\beta > 0$ there exists~${n_2 = n_2(k,\beta)}$
  such that the following holds for all~${n \geq n_2}$.
  If $\mathcal{H}$ is a linear $3$-graph on~$n$ vertices
  such that $B_\mathcal{H}$ has $\beta n^{3}$ dense components,
  each with at most $3^{10k^{2}}$ vertices,
  then~$\mathcal{H}$ contains a~$(k+3,k)$-configuration.
\end{prop}

\section{Proof of Theorem~\ref{thm:main}}
\label{sec:thm_main}
In short, we invoke Lemma~\ref{lemma:avg_degree} and deduce that the bow-tie graph graph $B$ has large average degree. Assuming it has no large component (as then we would be done by Proposition~\ref{lemma:large_config}), we deduce that $B$ must have many dense components. The assertion of Theorem~\ref{thm:main} then follows by Proposition~\ref{lemma:dense_config}.
\begin{proof}[Proof of Theorem~\ref{thm:main}]
  Let $1/5\geq \eps > 0$ and $k \geq 4$. Define an auxiliary constant
  $\beta = \eps/3^{11k^2}$.
  Apply Lemma~\ref{lemma:bowties} with $\delta = 4/5 + \eps$
  to obtain $n_1$, apply Proposition~\ref{lemma:dense_config} with~$k$ and~$\beta$ to
  obtain~$n_2$. Choose $n_3 =n_3(\eps)$ as in Lemma~\ref{lemma:avg_degree}, and define $n_0 = \max\{n_1,n_2,n_3\}$.
  Suppose that $\mathcal{H}$ is a linear $3$-graph
  on~$n\geq n_0$ vertices with~${\dlin(\mathcal{H}) \geq 4/5 + \eps}$.  By Lemma~\ref{lemma:avg_degree}, we have $$\davg(B)\geq 6+\eps.$$

    Let $C_1,\dots,C_\ell$ be the components of $B_\mathcal{H}$. If a component has least $3^{10k^2}$ vertices, we can directly apply Proposition~\ref{lemma:large_config}. Thus, let us assume that $|C_i|<3^{10k^2}$ for all $i$.

Let $I \subseteq [\ell]$ be the set of all $i \in [\ell]$ such that $C_i$ is dense.
  We have
  \begin{align*}
    \davg(B) = \sum_{i \in [\ell]}\frac{|C_i|}{|B|}\davg(C_i)
             = \sum_{i \in I}\frac{|C_i|}{|B|}\davg(C_i) + \sum_{i \in [\ell]\setminus I}\frac{|C_i|}{|B|}\davg(C_i).
  \end{align*}
  Using Lemma~\ref{lemma:eB}(i) we see that
  \begin{align}
    \davg(B) &= \sum_{i \in I}\frac{|C_i|}{|B|}\davg(C_i) + \sum_{i \in [\ell]\setminus I}\frac{|C_i|}{|B|}\davg(C_i)\nonumber\\
             &\leq \frac{8}{|B|}\sum_{i \in I}|C_i| + \frac{6}{|B|}\sum_{i \in [\ell]\setminus I}|C_i|\nonumber\\
             &= \frac{8}{|B|}\sum_{i \in I}|C_i| + \frac{6}{|B|}\Big(|B| - \sum_{i \in I}|C_i|\Big)\nonumber\\
             &= 6 + \frac{2}{|B|}\sum_{i \in I}|C_i| \label{eq:up_avg}.
  \end{align}
  
  \noindent
 Therefore,
  \[
    6 + \eps \leq \davg(B) \leq 6 + \frac{2}{|B|}\sum_{i \in I}|C_i|,
  \]
  which implies that
  \[
    \sum_{i \in I}|C_i| \geq \frac{\eps|B|}{2}.
  \]
  Since $|C_i| < 3^{10k^2}$ for all $i \in I$, it follows that
  \[
    |I| \geq \frac{\eps |B|}{3^{10k^2+1}}.
  \]
  By Lemma~\ref{lemma:bowties} we have 
  $$|B| \geq \frac{(4/5+\eps)^{2}}{16}n^3\geq \frac{n^3}{25}.$$
  Hence,
  \begin{align*}
    |I| \geq \frac{\eps |B|}{3^{10k^2+1}} \geq \frac{\eps}{3^{10k^2 + 1}}\cdot \frac{n^3}{25} \geq \frac{\eps}{3^{11k^2}}n^3 = \beta n^3.
  \end{align*}
  Thus, the bow-tie graph $B$ has at least $\beta n^3$ dense components, each
  with at most $3^{10k^2}$ vertices. By
  Proposition~\ref{lemma:dense_config}, $\mathcal{H}$ contains a $(k+3,3)$-configuration.
\end{proof}

\section{Concluding remarks}

We have shown that large linear $3$-graphs of density above $4/5$
contain $(k+3,k)$-configurations for any fixed $k$. We hope that subsequent papers
will gradually lower this density threshold. We believe this avenue could lead to progress towards a proof, or perhaps a disproof, of the
Brown-Erd\H{o}s-S\'os conjecture.

Of particular interest are the values $1/t\colon t=2,3,\dots$, since each of
them would imply a Ramsey version of the BESC studied in~\cite{ST21}: for
every $k\geq 3$ every $t$-colouring of a sufficiently large complete triple
system (conjecturally) contains a monochromatic $(k+3,k)$-configuration.  

Applying our method to the first open case $k=4$ of the BESC yields a
linear density threshold of $4/7$. Here again it would be interesting to try to
decrease it. 

\section*{Acknowledgements}

This project has received funding from the European Union’s Horizon 2020
research and innovation programme under the Marie Skłodowska-Curie grant
agreement No. 101007705.

\end{document}